\documentclass[11pt]{amsart}

\usepackage[T1]{fontenc}
\usepackage[utf8]{inputenc}
\usepackage{lmodern}
\usepackage{microtype}
\usepackage{amsmath,amssymb,amsthm,mathtools}
\usepackage{hyperref}
\usepackage{enumitem}
\usepackage{xcolor}

\hypersetup{
  colorlinks=true,
  linkcolor=[rgb]{0.10,0.10,0.45},
  citecolor=[rgb]{0.10,0.10,0.45},
  urlcolor=[rgb]{0.10,0.10,0.45}
}

\numberwithin{equation}{section}

\newtheorem{theorem}{Theorem}[section]
\newtheorem{proposition}[theorem]{Proposition}
\newtheorem{corollary}[theorem]{Corollary}
\newtheorem{lemma}[theorem]{Lemma}
\newtheorem{remark}[theorem]{Remark}
\newtheorem{example}[theorem]{Example}
\theoremstyle{definition}

\newcommand{\C}{\mathbb C}
\newcommand{\Prob}{\mathrm{Prob}}
\newcommand{\T}{\mathrm T}
\newcommand{\Ad}{\mathrm{Ad}}
\newcommand{\Aut}{\mathrm{Aut}}
\newcommand{\Inn}{\mathrm{Inn}}
\newcommand{\id}{\mathrm{id}}
\newcommand{\supp}{\mathrm{supp}}
\newcommand{\weakto}{\stackrel{w^*}{\longrightarrow}}
\newcommand{\ext}{\partial_{\mathrm e}}
\newcommand{\la}{\langle}
\newcommand{\ra}{\rangle}
\newcommand{\norm}[1]{\left\lVert #1 \right\rVert}
\newcommand{\normtwo}[2]{\left\lVert #1 \right\rVert_{2,#2}}
\newcommand{\abs}[1]{\left| #1 \right|}
\newcommand{\set}[1]{\left\{ #1 \right\}}
\newcommand{\restr}[2]{#1\!\upharpoonright_{#2}}
\newcommand{\Rad}{\mathrm R_{\mathrm a}}

\title[Invariant traces and relative property (T)]{Invariant trace simplices and relative property (T)}
\author{Raz Slutsky}
\address{Mathematical Institute, University of Oxford, Woodstock Road, Oxford OX2 6GG, United Kingdom}
\email{raz.slutsky@maths.ox.ac.uk}
\date{}

\begin{document}

\begin{abstract}
Let $\alpha\colon G\curvearrowright A$ be an action of a countable discrete group on a separable unital $C^*$-algebra. We study the simplex $\T(A)^G$ of $G$-invariant traces and ask when it is Bauer. Our main result is a noncommutative version of the Glasner--Weiss theorem: if $(G,H)$ has relative property~\textup{(T)} and the $H$-action on the von Neumann algebra of every extremal invariant trace is ergodic, that is, has only scalar fixed points, then $\T(A)^G$ is Bauer. 
We give criteria for the ergodicity hypothesis and apply them to certain quasi-local permutation actions, generalized Bernoulli actions, traces on group $C^*$-algebras, and reduced crossed products. In particular, if $G$ is infinite, has property~\textup{(T)}, and trivial amenable radical, then $C_r^*(\Delta\wr G)$ has Bauer trace simplex for every countable discrete group $\Delta$.
\end{abstract}

\maketitle

\section{Introduction}

A theorem of Glasner and Weiss says that if a countable discrete group $G$ with Kazhdan's property~\textup{(T)} acts by homeomorphisms on a compact metrizable space $X$, then the simplex $\Prob(X)^G$ of $G$-invariant Borel probability measures is Bauer, that is, it has closed extreme boundary, whenever it is nonempty \cite{GW}. Since
\[
\Prob(X)=\T(C(X)),
\]
this is exactly a statement about invariant traces on a commutative $C^*$-algebra. The purpose of this paper is to develop a noncommutative counterpart.

If $\alpha\colon G\curvearrowright A$ is an action on a separable unital $C^*$-algebra, we write
\[
\T(A)^G:=\set{\tau\in \T(A):\tau\circ \alpha_g=\tau\text{ for all }g\in G}.
\]

In the noncommutative setting, property~\textup{(T)} alone cannot force $\T(A)^G$ to be Bauer. Indeed, under the trivial action one has $\T(A)^G=\T(A)$, and trace simplices of separable unital $C^*$-algebras can be highly non-Bauer. For example, $\T(C^*(F_d))$, and more generally the trace spaces of most full free products, are Poulsen simplices\footnote{A Poulsen simplex is the unique simplex with the property that its extreme points are dense.}  \cite{OSV,ISV}. In fact, every metrizable Choquet simplex occurs as the trace space of a simple AF algebra \cite{Blackadar}. Thus any noncommutative Glasner--Weiss theorem must incorporate genuinely dynamical input.

For $\tau\in \T(A)^G$, let $M_\tau=\pi_\tau(A)''$ be the tracial von Neumann algebra in the GNS representation of $\tau$. Since $\tau$ is $G-$invariant, $\alpha$ extends to a trace-preserving action on $M_\tau$. If $H\leq G$, we say that the restricted $H$-action on $M_\tau$ is \emph{ergodic} if
\[
M_\tau^H=\C.
\]
In the commutative case $A=C(X)$ and $M_\tau=L^\infty(X,\mu)$, this is just the usual ergodicity of the corresponding measure-preserving action. See for example \cite{AHK,HLS} for this terminology in the operator-algebraic setting.

Our main theorem shows that relative property~\textup{(T)} implies Bauerness once ergodicity is assumed. We write $\ext \T(A)^G:=\partial_{\mathrm e}(\T(A)^G)$ for the extreme boundary of the
invariant trace simplex $\T(A)^G$.

\begin{theorem}
\label{thm:intro-main}
 Let $H\leq G$. If the pair $(G,H)$ has relative property~\textup{(T)} and $H$ acts ergodically on $M_\tau$ for all \(\tau\in \ext \T(A)^G\),
then $\T(A)^G$ is a Bauer simplex.
\end{theorem}

Taking $H=G$ and $A=C(X)$ recovers the theorem of Glasner and Weiss. When $A=C^*(G)$ and $H=G$ acts by conjugation, the scalar fixed-point hypothesis is automatic on extremal traces, and one recovers that for groups with property~\textup{(T)}, the simplex of traces on the group, or equivalently, on its maximal $C^*$-algebra is Bauer, see \cite{LSV}.

One reason the Bauer conclusion is useful is that it gives a compact topological
parametrization of ergodic invariant traces. Indeed, if $K$ is a Bauer simplex,
then $\partial_{\mathrm e}K$ is compact and $K$ is canonically affinely
homeomorphic to $\Prob(\partial_{\mathrm e}K)$. This regularity also fits naturally with existing structure theory. Compact
extreme boundary is the natural topological setting for packaging tracial GNS
completions into continuous bundles of tracial von Neumann algebras in the
sense of Ozawa \cite{OzawaWbundles}; see also \cite{CCEGSTW}. In the dynamical
direction, assumptions on Bauerness are useful in proving certain regularity
results, see for example \cite{GGNV, KKP}.

To make Theorem~\ref{thm:intro-main} useful, we prove some criteria forcing scalar or central fixed points, based on displacement, innerness, and mixing. The guiding principle is that once the action sees enough of the algebra, the
invariant trace space becomes Bauer. We start with the following consequence.

\begin{theorem}[Inner generation criterion]
\label{thm:intro-inner}
Assume that $G$ has property~\textup{(T)}. Suppose that there is a subset
$S\subseteq G$ such that $\alpha_s$ is inner for every $s\in S$, and for some
choice of implementing unitaries $u_s\in U(A)$ one has
\[
A=C^*\bigl(Z(A)\cup \{u_s:s\in S\}\bigr).
\]
Then $\T(A)^G$ is Bauer.
\end{theorem}

This generalises \cite[Thm.~1.9]{LSV}, and shows, for example, that
\(\T(C^*(F_n))^{\Aut(F_n)}\) is Bauer for every \(n\geq 4\), using property
\textup{(T)} of \(\Aut(F_n)\) due to
\cite{NitscheAutF4T,KNOAutF5T,KKNAutFnT}.

We also get concrete new applications from quasi-local permutation actions. In particular, the Bernoulli case gives an interesting example:

\begin{corollary}[Bernoulli embedding]
Let $G$ be an infinite countable discrete group with property ~\textup{(T)} , let $D$ be a separable unital $C^*$-algebra with $\T(D)\neq\varnothing$, and let
\[
A=\bigotimes_G D
\]
with the Bernoulli shift action of $G$. Then $\T(A)^G$ is Bauer, and the map
\[
\T(D)\longrightarrow \ext \T(A)^G,
\qquad
\rho\longmapsto \bigotimes_G \rho,
\]
is a continuous injection.
\end{corollary}

Thus, even when $\T(A)^G$ is Bauer, its extreme boundary can be quite large. For example, taking $D=C^*(F_2)$ shows that $\ext \T(A)^G$ contains a compact space homeomorphic to the Poulsen simplex \cite{OSV}.

Finally, we obtain concrete crossed-product consequences. 

\begin{corollary}[Permutational wreath products]
Let $I$ be a countable $G$-set and let $\Delta$ be a countable discrete group. Assume that $G$ has property~\textup{(T)}, that $\Rad(G)=\{e\}$, and that the action of $G$ on $I$ displaces finite subsets. Then
\[
C_r^*\bigl((\bigoplus_I \Delta)\rtimes G\bigr)
\]
has Bauer trace simplex. In particular, if $G$ is infinite and $I=G$ with the left translation action, then
\[
C_r^*(\Delta\wr G)
\]
has Bauer trace simplex.
\end{corollary}

The crossed-product applications use external trace-extension theorems.
Ursu provides a very general framework for traces on crossed products \cite{Ursu},
and in the reduced case with trivial amenable radical,
Bryder--Kennedy identifies $\T(A\rtimes_{\alpha,r}G)$ with $\T(A)^G$
\cite{BryderKennedy}. We combine these correspondences with the Bauer results
proved above to obtain new examples.

We remark generally that the hypotheses are not meant to characterise when $\T(A)^G$ is Bauer: even
nontrivial property~\textup{(T)} actions on algebras of the form $B\oplus B$
can have Bauer invariant trace space while the fixed-point algebra on an
extremal invariant trace is far from scalar.

A related line of work considers invariant states. Since the space of states is usually not a simplex, this was
developed in deep work by Kennedy and Shamovich, introducing the theory of non-commutative Choquet simplices. This allows them to show that for an action of a discrete property~\textup{(T)} group on a unital $C^*$-algebra, the invariant
nc state space is an nc Bauer simplex, equivalently affinely homeomorphic to
the nc state space of a unital $C^*$-algebra \cite[\S14]{KS}; see also
\cite{KKM} and the survey \cite{davidson2024noncommutative}. Our focus here is the ordinary Choquet simplex $\T(A)^G$ of invariant tracial states.

The proof of Theorem~\ref{thm:intro-main} has two ingredients. First, we identify the face generated by an invariant trace $\tau$ with the normalized positive cone of $Z(M_\tau)^G$. Second, once extremality is recast in von Neumann algebraic terms, the Bauer problem becomes a uniform spectral-gap problem. Relative property~\textup{(T)} supplies the required gap precisely when subgroup fixed points are scalar.

The paper is organized as follows. Section~\ref{sec:preliminaries} studies the simplex $\T(A)^G$ and proves the fixed-center description of faces and extremal traces. Section~\ref{sec:main} establishes the uniform spectral-gap theorem and the relative property~\textup{(T)} criterion above. Section~\ref{sec:criteria} develops the displacement, innerness, and mixing criteria for scalar or central fixed points. Section~\ref{sec:applications} applies these criteria to quasi-local permutation actions, generalized Bernoulli actions, invariant traces on maximal and reduced group $C^*$-algebras, and reduced crossed products arising from wreath products and semidirect products.

\subsection*{Acknowledgements}
We thank Itamar Vigdorovich for his very helpful comments on a draft of this paper.

\section{Preliminaries}
\label{sec:preliminaries}

Throughout the paper, $A$ denotes a separable unital $C^*$-algebra, $G$ a countable discrete group, and $\alpha\colon G\to \Aut(A)$ an action.
For $\tau\in \T(A)$ we write
\[
(\pi_\tau,H_\tau,\xi_\tau)
\]
for the tracial GNS triple and $M_\tau=\pi_\tau(A)''$ for the associated finite von Neumann algebra.
The trace on $A$ extends to a faithful normal tracial state on $M_\tau$, still denoted by $\tau$.
We equip $M_\tau$ with the $2$-norm
\[
\normtwo{x}{\tau}:=\tau(x^*x)^{1/2}.
\]
Via the identification $H_\tau=L^2(M_\tau,\tau)$ we also use the right multiplication operators
\[
R_x(y\xi_\tau):=yx\xi_\tau
\qquad (x,y\in M_\tau),
\]
so that $R(M_\tau)=M_\tau'$.

If $\tau\in \T(A)^G$, the action of $G$ on $A$ extends and is unitarily implemented on $H_\tau$ by
\[
U^\tau_g\pi_\tau(a)\xi_\tau=\pi_\tau(\alpha_g(a))\xi_\tau,
\qquad a\in A,\ g\in G.
\]
The same symbol $\alpha$ will be used for the induced trace-preserving action on $M_\tau$. See \cite[Theorem~2.3.16 and Corollary~2.3.17]{BR1} for details.

\subsection{The simplex \texorpdfstring{$\T(A)^G$}{invariant trace simplex}}

It is classical that $\T(A)$ is a metrizable Choquet simplex; see for example \cite{BR}.
We record that the invariant trace space is again a metrizable Choquet simplex.

\begin{proposition}
\label{prop:fixed-simplex}
If $\T(A)^G\neq\varnothing$, then $\T(A)^G$ is a metrizable Choquet simplex.
\end{proposition}

\begin{proof}
Let $X=\ext\T(A)$, equipped with its Borel structure.
Since $\T(A)$ is a metrizable Choquet simplex, every $\tau\in \T(A)$ admits a unique representing probability measure $\mu_\tau$ on $X$.The action of $G$ on $\T(A)$ by
\[
g\cdot \tau:=\tau\circ \alpha_{g^{-1}}
\]
preserves $X$.
If $\tau\in \T(A)$ and $g\in G$, then $g_*\mu_\tau$ is a representing measure for $g\cdot \tau$, so uniqueness gives
\[
\mu_{g\cdot \tau}=g_*\mu_\tau.
\]
Hence $\tau\in \T(A)^G$ if and only if $\mu_\tau$ is $G$-invariant.
Therefore the barycenter map identifies $\T(A)^G$ affinely with the set of $G$-invariant Borel probability measures on the standard Borel $G$-space $X$.
By the ergodic decomposition theorem for countable group actions on standard Borel spaces \cite{Varadarajan}, the latter is a Choquet simplex.
Metrizability follows from weak-$^*$ compactness and the separability of $A$.
\end{proof}

\subsection{Faces and extremality}

The next proposition identifies the smallest face of $\T(A)^G$ containing a given invariant trace, which we will call the face generated by it.

\begin{proposition}
\label{prop:trace-face}
Let $\tau\in \T(A)^G$.
Then the face of $\T(A)^G$ generated by $\tau$ is
\[
\mathcal F(\tau)
=
\set{\sigma\in \T(A)^G:\ \sigma\leq c\tau\ \text{for some }c>0}.
\]
Moreover, for $\sigma\in \T(A)^G$ the following are equivalent:
\begin{enumerate}[label=\textup{(\roman*)}]
\item $\sigma\in \mathcal F(\tau)$;
\item there exists $h\in Z(M_\tau)^G_+$ with $\tau(h)=1$ such that
\[
\sigma(a)=\tau\bigl(h\pi_\tau(a)\bigr)
\qquad (a\in A).
\]
\end{enumerate}
In particular, $\mathcal F(\tau)$ is affinely isomorphic to the convex set
\[
\set{h\in Z(M_\tau)^G_+:\ \tau(h)=1}.
\]
\end{proposition}

\begin{proof}
Set
\[
F_0:=\set{\sigma\in \T(A)^G:\ \sigma\leq c\tau\ \text{for some }c>0}.
\]
We first show that $F_0$ is the face generated by $\tau$.
It is immediate that $\tau\in F_0$.
If
\[
\sigma=t\sigma_1+(1-t)\sigma_2\in F_0
\qquad (0<t<1)
\]
and $\sigma\leq c\tau$, then
\[
\sigma_1\leq \frac{c}{t}\tau,
\qquad
\sigma_2\leq \frac{c}{1-t}\tau,
\]
so $\sigma_1,\sigma_2\in F_0$.
Thus $F_0$ is a face.

Now let $F$ be any face of $\T(A)^G$ containing $\tau$, and let $\sigma\in F_0$.
Choose $c>1$ with $\sigma\leq c\tau$ and define
\[
\eta:=\frac{c\tau-\sigma}{c-1}\in \T(A)^G.
\]
Then
\[
\tau=\frac1c\,\sigma+\Bigl(1-\frac1c\Bigr)\eta.
\]
Since $F$ is a face and $\tau\in F$, it follows that $\sigma\in F$.
Therefore $F_0$ is the smallest face containing $\tau$.

Assume now that $\sigma\in \mathcal F(\tau)$.
Choose $c>0$ with $\sigma\leq c\tau$ and set $\rho:=c^{-1}\sigma$.
Then $0\leq \rho\leq \tau$.
By the Radon--Nikod\'ym theorem for positive functionals dominated by a state in the GNS representation \cite[Chapter~1, \S24]{Sakai}, there exists a positive contraction $T\in \pi_\tau(A)'$ such that
\[
\rho(a)=\la T\pi_\tau(a)\xi_\tau,\xi_\tau\ra
\qquad (a\in A).
\]
Since $H_\tau=L^2(M_\tau,\tau)$ is tracial, $T=R_k$ for a unique positive contraction $k\in M_\tau$.
Hence
\[
\rho(a)=\tau\bigl(k\pi_\tau(a)\bigr)
\qquad (a\in A).
\]

We claim that $k\in Z(M_\tau)^G$.
Fix $a,b\in A$.
Since $\rho$ is a trace on $A$ and $\tau$ is a trace on $M_\tau$,
\begin{align*}
\tau\bigl((k\pi_\tau(a)-\pi_\tau(a)k)\pi_\tau(b)\bigr)
&=\tau(k\pi_\tau(a)\pi_\tau(b))-\tau(\pi_\tau(a)k\pi_\tau(b))\\
&=\rho(ab)-\tau(k\pi_\tau(b)\pi_\tau(a))\\
&=\rho(ab)-\rho(ba)=0.
\end{align*}
For fixed $a$, the map
\[
x\longmapsto \tau\bigl((k\pi_\tau(a)-\pi_\tau(a)k)x\bigr)
\]
is a normal functional on $M_\tau$ vanishing on the $\sigma$-weakly dense $*$-subalgebra $\pi_\tau(A)$.
Hence it vanishes on all of $M_\tau$.
Taking $x=(k\pi_\tau(a)-\pi_\tau(a)k)^*$ and using faithfulness of $\tau$, we get
\[
k\pi_\tau(a)=\pi_\tau(a)k
\qquad (a\in A).
\]
Thus $k\in Z(M_\tau)$.

Now let $g\in G$ and $a\in A$.
Using the $G$-invariance of $\tau$ and $\rho$, we obtain
\begin{align*}
\tau\bigl((\alpha_g(k)-k)\pi_\tau(a)\bigr)
&=\tau\bigl(\alpha_g(k)\pi_\tau(a)\bigr)-\tau\bigl(k\pi_\tau(a)\bigr)\\
&=\tau\bigl(\alpha_g(k\pi_\tau(\alpha_{g^{-1}}(a)))\bigr)-\rho(a)\\
&=\tau\bigl(k\pi_\tau(\alpha_{g^{-1}}(a))\bigr)-\rho(a)\\
&=\rho(\alpha_{g^{-1}}(a))-\rho(a)=0.
\end{align*}
As before, normality and $\sigma$-weak density imply $\alpha_g(k)=k$.
Thus $k\in Z(M_\tau)^G$.

Set $h:=ck$.
Then $h\in Z(M_\tau)^G_+$, $\tau(h)=\sigma(1)=1$, and
\[
\sigma(a)=\tau\bigl(h\pi_\tau(a)\bigr)
\qquad (a\in A).
\]
This proves \textup{(i)}$\Rightarrow$\textup{(ii)}.

Conversely, if $h\in Z(M_\tau)^G_+$ with $\tau(h)=1$, then
\[
\sigma_h(a):=\tau\bigl(h\pi_\tau(a)\bigr)
\]
defines a $G$-invariant trace on $A$.
Since $0\leq h\leq \norm{h}1$, we have $\sigma_h\leq \norm{h}\tau$, so $\sigma_h\in \mathcal F(\tau)$.
This proves \textup{(ii)}$\Rightarrow$\textup{(i)}.
The final statement is immediate from the explicit parametrization.
\end{proof}

\begin{corollary}
\label{cor:trace-extreme-criterion}
Let $\tau\in \T(A)^G$.
Then
\[
\tau\in \ext \T(A)^G
\quad\Longleftrightarrow\quad
Z(M_\tau)^G=\C1.
\]
\end{corollary}

\begin{proof}
By Proposition~\ref{prop:trace-face}, $\tau$ is extreme if and only if the face $\mathcal F(\tau)$ is reduced to $\{\tau\}$.
This happens exactly when the only element $h\in Z(M_\tau)^G_+$ with $\tau(h)=1$ is $h=1$.
Equivalently, $Z(M_\tau)^G=\C1$.
\end{proof}

\begin{remark}
\label{rem:classical-antecedents}
Compare to Robinson--Ruelle and St{\o}rmer \cite{RobinsonRuelle,Stormer} in the invariant-states setting. 
\end{remark}

\begin{remark}
\label{rem:closed-face-warning}
Proposition~\ref{prop:trace-face} identifies the \emph{algebraic} face generated by $\tau$, not its weak-$^*$ closure.
Already for $A=C(X)$ and a nonatomic measure $\mu\in \Prob(X)$, the face generated by $\mu$ consists of the probability measures with essentially bounded density with respect to $\mu$, whereas the closed face generated by $\mu$ is the full simplex $\Prob(\supp\mu)$.
\end{remark}

\subsection{Fixed vectors and relative property~\textup{(T)}}
We will first need the following lemma.
\begin{lemma}
\label{lem:l2-fixed}
Let $H\leq G$ and $\tau\in \T(A)^G$.
Then
\[
H_\tau^{U^\tau(H)}=L^2(M_\tau^H,\tau).
\]
In particular,
\[
M_\tau^H=\C
\quad\Longleftrightarrow\quad
H_\tau^{U^\tau(H)}=\C\xi_\tau.
\]
\end{lemma}

\begin{proof}
Set $N=M_\tau^H$, and let $P$ be the orthogonal projection of $H_\tau$ onto
$H_\tau^{U^\tau(H)}$. If $y\in N$, then for every $h\in H$,
\[
U_h^\tau(y\xi_\tau)=\alpha_h(y)\xi_\tau=y\xi_\tau,
\]
so $N\xi_\tau\subseteq H_\tau^{U^\tau(H)}$. Since $N\xi_\tau$ is dense in
$L^2(N,\tau)$, it follows that
\[
L^2(N,\tau)\subseteq H_\tau^{U^\tau(H)}.
\]
It remains to prove the reverse inclusion.

Fix $x\in M_\tau$, and set
\[
C:=\overline{\operatorname{conv}}\{U_h^\tau(x\xi_\tau):h\in H\}\subseteq H_\tau.
\]
Then $C$ is a nonempty closed convex $U^\tau(H)$-invariant subset of $H_\tau$.
Let $\eta_0\in C$ be the unique element of minimal norm. Since each
$U_h^\tau$ is unitary and $C$ is $U^\tau(H)$-invariant, uniqueness implies
that $\eta_0$ is $H$-fixed.

For each $h\in H$ we have
\[
P(U_h^\tau(x\xi_\tau))=P(x\xi_\tau),
\]
so $P$ is constant on the convex hull of the orbit, and hence, by continuity,
on $C$. Applying this to $\eta_0$ yields
\(
\eta_0=P\eta_0=P(x\xi_\tau),
\)
and therefore
\[
P(x\xi_\tau)\in \overline{\operatorname{conv}}\{U_h^\tau(x\xi_\tau):h\in H\}.
\]

Thus there exists a net $(y_i)$ in
$\operatorname{conv}\{\alpha_h(x):h\in H\}\subseteq M_\tau$ such that
\[
\norm{y_i\xi_\tau-P(x\xi_\tau)}\longrightarrow 0.
\]
Since $\norm{y_i}\leq \norm{x}$ for all $i$, the net $(y_i)$ is bounded in
$M_\tau$. Passing to a subnet, we may assume that $y_i\to y$ $\sigma$-weakly
for some $y\in M_\tau$. Then for every $z\in M_\tau$,
\[
\langle y_i\xi_\tau,z\xi_\tau\rangle=\tau(z^*y_i)\longrightarrow \tau(z^*y)
=\langle y\xi_\tau,z\xi_\tau\rangle,
\]
so $y_i\xi_\tau\to y\xi_\tau$ weakly in $L^2(M_\tau,\tau)$. Since
$y_i\xi_\tau\to P(x\xi_\tau)$ in norm, it follows that
\(
P(x\xi_\tau)=y\xi_\tau.
\) Because $P(x\xi_\tau)$ is $H$-fixed, for every $h\in H$,
\[
\alpha_h(y)\xi_\tau
=
U_h^\tau(y\xi_\tau)
=
U_h^\tau P(x\xi_\tau)
=
P(x\xi_\tau)
=
y\xi_\tau.
\]
Faithfulness of $\tau$ implies $\alpha_h(y)=y$, so $y\in N$. Hence
\[
P(x\xi_\tau)\in L^2(N,\tau).
\]

We have shown that $P(M_\tau\xi_\tau)\subseteq L^2(N,\tau)$. Since
$M_\tau\xi_\tau$ is dense in $H_\tau$, $P$ is continuous, and $L^2(N,\tau)$ is
closed, it follows that
\[
H_\tau^{U^\tau(H)}=\operatorname{ran}P\subseteq L^2(N,\tau).
\]
Combining this with the inclusion proved at the beginning gives
\[
H_\tau^{U^\tau(H)}=L^2(M_\tau^H,\tau).
\]
The final equivalence is immediate.
\end{proof}
Finally, we will need the following form of relative property~\textup{(T)}.

\begin{proposition}
\label{prop:relative-kazhdan-pair}
Let $H\leq G$.
If the pair $(G,H)$ has relative property~\textup{(T)}, then there exist a finite set $F\subseteq G$ and $\kappa>0$ such that for every unitary representation $V\colon G\to \mathcal U(K)$ satisfying that $K^{V(H)}$ is $G$-invariant, and every vector $\eta\in K\ominus K^{V(H)}$,
\[
\max_{g\in F}\norm{V_g\eta-\eta}\geq \kappa\norm{\eta}.
\]
\end{proposition}

\begin{proof}
Choose a Kazhdan pair $(F,\kappa)$ for $(G,H)$; see \cite[Definition~1.4.3 and Remark~1.4.4]{BHV}.
Since $K^{V(H)}$ is $G$-invariant, the orthogonal complement $K\ominus K^{V(H)}$ is a $G$-invariant subspace.
The restricted representation on $K\ominus K^{V(H)}$ has no nonzero $H$-fixed vectors.
If the inequality above failed for some nonzero $\eta\in K\ominus K^{V(H)}$, then $\eta/\norm{\eta}$ would be $(F,\kappa)$-almost invariant for a representation with no nonzero $H$-fixed vectors, contradicting the defining property of $(F,\kappa)$.
\end{proof}

\section{Uniform spectral gap and Bauer simplices}
\label{sec:main}

We now prove the basic spectral-gap theorem.

\begin{theorem}
\label{thm:uniform-gap}
Assume that there exist a finite set $F\subseteq G$ and $\varepsilon>0$ such that for every $\tau\in \ext \T(A)^G$ and every $a\in A$,
\begin{equation}
\label{eq:SG}
\max_{g\in F}\normtwo{\alpha_g(a)-a}{\tau}
\geq
\varepsilon \inf_{\lambda\in \C}\normtwo{a-\lambda1}{\tau}.
\end{equation}
Then $\ext \T(A)^G$ is weak-$^*$ closed.
Equivalently, $\T(A)^G$ is a Bauer simplex.
\end{theorem}

\begin{proof}
Let $(\tau_n)$ be a sequence in $\ext \T(A)^G$ converging weak-$^*$ to some $\tau\in \T(A)^G$.
Suppose, for contradiction, that $\tau\notin \ext \T(A)^G$.
By Corollary~\ref{cor:trace-extreme-criterion},
\[
Z(M_\tau)^G\neq \C1.
\]
Choose $z\in Z(M_\tau)^G$ self-adjoint and non-scalar.
Replacing $z$ by $z-\tau(z)1$ and rescaling, we may assume
\[
\tau(z)=0,
\qquad
\normtwo{z}{\tau}=1.
\]

Fix $\delta>0$.
Since $\pi_\tau(A)\xi_\tau$ is dense in $H_\tau=L^2(M_\tau,\tau)$, there exists a self-adjoint $a\in A$ such that
\(
\normtwo{a-z}{\tau}<\delta.
\)
Then
\[
\abs{\tau(a)}\leq \normtwo{a-z}{\tau}<\delta,
\qquad
1-\delta<\normtwo{a}{\tau}<1+\delta.
\]
Because $z$ is $G$-fixed,
\[
\normtwo{\alpha_g(a)-a}{\tau}
\leq \normtwo{\alpha_g(a)-z}{\tau}+\normtwo{z-a}{\tau}
=2\normtwo{a-z}{\tau}<2\delta
\]
for every $g\in F$.

Since $F$ is finite and $\tau_n\weakto \tau$, for all sufficiently large $n$ we have
\begin{align*}
\abs{\tau_n(a)}&<2\delta,\\
1-2\delta&<\normtwo{a}{\tau_n}<1+2\delta,\\
\max_{g\in F}\normtwo{\alpha_g(a)-a}{\tau_n}&<3\delta.
\end{align*}
Moreover,
\[
\inf_{\lambda\in \C}\normtwo{a-\lambda1}{\tau_n}^2
=\normtwo{a}{\tau_n}^2-\abs{\tau_n(a)}^2.
\]
Hence, for $\delta$ small enough and all large $n$,
\[
\inf_{\lambda\in \C}\normtwo{a-\lambda1}{\tau_n}\geq \frac12.
\]
Applying \eqref{eq:SG} to $\tau_n$ and $a$, we obtain
\[
3\delta
>
\max_{g\in F}\normtwo{\alpha_g(a)-a}{\tau_n}
\geq
\varepsilon\inf_{\lambda\in \C}\normtwo{a-\lambda1}{\tau_n}
\geq \frac{\varepsilon}{2},
\]
which is impossible for $\delta<\varepsilon/6$.
Therefore $\tau\in \ext\T(A)^G$, as required.
\end{proof}

We now show that this spectral-gap hypothesis is automatic under relative property~\textup{(T)} once one knows that subgroup fixed points are scalar.

\begin{corollary}[Relative noncommutative Glasner--Weiss]
\label{cor:main-relative}
Let $H\leq G$.
Assume that the pair $(G,H)$ has relative property~\textup{(T)} and that
\[
M_\tau^H=\C
\qquad\text{for every }\tau\in \ext \T(A)^G.
\]
Then $\T(A)^G$ is a Bauer simplex.
\end{corollary}

\begin{proof}
By Lemma~\ref{lem:l2-fixed}, the hypothesis is equivalent to
\[
H_\tau^{U^\tau(H)}=\C\xi_\tau
\qquad (\tau\in \ext \T(A)^G).
\]
Let $F\subseteq G$ and $\kappa>0$ be as in Proposition~\ref{prop:relative-kazhdan-pair}.
Fix $\tau\in \ext \T(A)^G$ and $a\in A$.
Set
\[
\eta:=\pi_\tau(a)\xi_\tau-\tau(a)\xi_\tau.
\]
Then $\eta\perp \xi_\tau$, hence $\eta\in H_\tau\ominus H_\tau^{U^\tau(H)}$.
Since $\C\xi_\tau$ is $G$-invariant, Proposition~\ref{prop:relative-kazhdan-pair} applies to the representation $U^\tau$.
For $g\in G$,
\[
U^\tau_g\eta-\eta=\pi_\tau(\alpha_g(a)-a)\xi_\tau,
\]
so Proposition~\ref{prop:relative-kazhdan-pair} gives
\[
\max_{g\in F}\normtwo{\alpha_g(a)-a}{\tau}
=\max_{g\in F}\norm{U^\tau_g\eta-\eta}
\geq \kappa\norm{\eta}.
\]
Since
\[
\norm{\eta}=\inf_{\lambda\in \C}\normtwo{a-\lambda1}{\tau},
\]
condition \eqref{eq:SG} holds.
The conclusion follows from Theorem~\ref{thm:uniform-gap}.
\end{proof}

In particular, by taking $H = G$, one gets:

\begin{corollary}[Noncommutative Glasner--Weiss for traces]
\label{cor:main-trace}
Assume that $G$ has property~\textup{(T)}, and that
\[
M_\tau^G=\C
\qquad\text{for every }\tau\in \ext \T(A)^G.
\]
Then $\T(A)^G$ is a Bauer simplex.
\end{corollary}

\begin{remark}
\label{rem:central-detection}
Let $\tau\in \ext \T(A)^G$.
Since Corollary~\ref{cor:trace-extreme-criterion} gives $Z(M_\tau)^G=\C$, the condition $M_\tau^G=\C$ may be replaced by the weaker inclusion
\[
M_\tau^G\subseteq Z(M_\tau).
\]
This is often the most convenient form in applications.
\end{remark}

\section{Criteria forcing scalar or central fixed points}
\label{sec:criteria}

This section gives three concrete mechanisms implying the hypotheses of Corollary~\ref{cor:main-relative} or Corollary~\ref{cor:main-trace}.

\subsection{A displacement criterion}

\begin{proposition}
\label{prop:displacement-central}
Assume that there exists a norm-dense $*$-subalgebra $A_{\mathrm{loc}}\subseteq A$ with the following property:
for every finite sets $\Omega,\mathcal K\subseteq A_{\mathrm{loc}}$, there exists $g\in G$ such that
\[
[\alpha_g(a),b]=0
\qquad (a\in \Omega,\ b\in \mathcal K).
\]
Then, for every $\tau\in \T(A)^G$,
\[
M_\tau^G\subseteq Z(M_\tau).
\]
Consequently, if $G$ has property~\textup{(T)} then $\T(A)^G$ is a Bauer simplex.
\end{proposition}

\begin{proof}
Fix $\tau\in \T(A)^G$ and let $x\in M_\tau^G$.
We show that $x$ commutes with $\pi_\tau(A_{\mathrm{loc}})$.
Let $b\in A_{\mathrm{loc}}$ and choose a sequence $(a_n)$ in $A_{\mathrm{loc}}$ such that
\[
\normtwo{x-\pi_\tau(a_n)}{\tau}\longrightarrow 0.
\]
For each $n$, apply the hypothesis to $\Omega=\set{a_n}$ and $\mathcal K=\set{b}$ to obtain $g_n\in G$ with
\[
[\alpha_{g_n}(a_n),b]=0.
\]
Therefore
\[
[\pi_\tau(\alpha_{g_n}(a_n)),\pi_\tau(b)]=0.
\]
Since $x$ is $G$-fixed and $\tau$ is $G$-invariant,
\begin{align*}
\normtwo{[x,\pi_\tau(b)]}{\tau}
&=\normtwo{[\alpha_{g_n}(x),\pi_\tau(b)]}{\tau}\\
&=\normtwo{[\alpha_{g_n}(x)-\pi_\tau(\alpha_{g_n}(a_n)),\pi_\tau(b)]}{\tau}\\
&\leq 2\norm{b}\,\normtwo{x-\pi_\tau(a_n)}{\tau}.
\end{align*}
Letting $n\to\infty$ gives $[x,\pi_\tau(b)]=0$.
Since $\pi_\tau(A_{\mathrm{loc}})$ is $\sigma$-weakly dense in $M_\tau$, it follows that $x\in Z(M_\tau)$.
This proves the stated inclusion.

If $G$ has property~\textup{(T)} then Remark~\ref{rem:central-detection} shows that
\[
M_\tau^G=\C
\qquad (\tau\in \ext \T(A)^G),
\]
and Corollary~\ref{cor:main-trace} yields that $\T(A)^G$ is Bauer.
\end{proof}

\begin{corollary}[Quasi-local permutation actions]
\label{cor:quasi-local}
Let $I$ be a countable $G$-set.
Assume that $A$ contains a norm-dense $*$-subalgebra
\[
A_{\mathrm{loc}}=\bigcup_{F\subseteq I\text{ finite}} A_F
\]
such that:
\begin{enumerate}[label=\textup{(\roman*)}]
\item $A_F\subseteq A_K$ whenever $F\subseteq K$;
\item $\alpha_g(A_F)=A_{gF}$ for every finite $F\subseteq I$ and every $g\in G$;
\item $[A_F,A_K]=0$ whenever $F\cap K=\varnothing$.
\end{enumerate}
Assume moreover that the action of $G$ on $I$ displaces finite subsets, in the sense that for every finite $F,K\subseteq I$ there exists $g\in G$ with $gF\cap K=\varnothing$.
Then, for every $\tau\in \T(A)^G$,
\[
M_\tau^G\subseteq Z(M_\tau).
\]
Consequently, if $G$ has property~\textup{(T)} then $\T(A)^G$ is a Bauer simplex.
\end{corollary}

\begin{proof}
It suffices to verify the hypothesis of Proposition~\ref{prop:displacement-central}.
Let $\Omega,\mathcal K\subseteq A_{\mathrm{loc}}$ be finite.
By \textup{(i)}, there exist finite subsets $F,K\subseteq I$ such that
\[
\Omega\subseteq A_F,
\qquad
\mathcal K\subseteq A_K.
\]
By the displacement assumption, choose $g\in G$ with $gF\cap K=\varnothing$.
Then \textup{(ii)} and \textup{(iii)} imply
\[
[\alpha_g(a),b]=0
\qquad (a\in \Omega,\ b\in \mathcal K).
\]
Proposition~\ref{prop:displacement-central} now applies.
\end{proof}

\subsection{An inner-core criterion}

\begin{proposition}
\label{prop:inner-core}
Assume that there exist a subgroup $H\leq G$ and unitaries $(u_h)_{h\in H}\subseteq U(A)$ such that
\[
\alpha_h=\Ad(u_h)
\qquad (h\in H).
\]
Suppose furthermore that for every $\tau\in \ext \T(A)^G$ one has
\[
W^*(\pi_\tau(u_h):h\in H)'\cap M_\tau\subseteq Z(M_\tau).
\]
Then
\[
M_\tau^G=\C
\qquad\text{for every }\tau\in \ext \T(A)^G.
\]
In particular, if $G$ has property~\textup{(T)}, then $\T(A)^G$ is a Bauer simplex.
\end{proposition}

\begin{proof}
Fix $\tau\in \ext \T(A)^G$ and let $x\in M_\tau^G$.
Then $x$ is fixed in particular by the subgroup $H$.
Since the action of $h\in H$ on $M_\tau$ is implemented by $\pi_\tau(u_h)$, we have
\[
\pi_\tau(u_h)x\pi_\tau(u_h)^*=x
\qquad (h\in H).
\]
Thus
\[
x\in W^*(\pi_\tau(u_h):h\in H)'\cap M_\tau\subseteq Z(M_\tau).
\]
Because $x$ is also $G$-fixed, Corollary~\ref{cor:trace-extreme-criterion} gives
\[
x\in Z(M_\tau)^G=\C1.
\]
Hence $M_\tau^G=\C$.
The final assertion follows from Corollary~\ref{cor:main-trace}.
\end{proof}

This proves Theorem \ref{thm:intro-inner} by taking $H = \langle S \rangle.$

\begin{remark}
\label{rem:inner-special-case}
If the implementing unitaries generate $M_\tau$ for every $\tau\in \ext \T(A)^G$, then the commutant hypothesis in Proposition~\ref{prop:inner-core} is automatic, since
\[
W^*(\pi_\tau(u_h):h\in H)'\cap M_\tau=M_\tau'\cap M_\tau=Z(M_\tau).
\]
In the special case where $H=G$, the action is implemented by a unitary representation $u\colon G\to U(A)$ and $A=C^*(u_g:g\in G)$, the Bauer conclusion also follows formally from the fact that $A$ is a quotient of $C^*(G)$ and that $\T(C^*(G))$ is Bauer for property~\textup{(T)} groups \cite{LSV}. Proposition~\ref{prop:inner-core} is useful because it also applies when only a subgroup acts innerly or when the implementing unitaries do not generate $A$.
\end{remark}

\subsection{A mixing criterion}

\begin{proposition}
\label{prop:mixing}
Let $H\leq G$ be infinite, and let $\tau\in \T(A)^G$.
Assume that for all $x,y\in M_\tau$ with $\tau(x)=\tau(y)=0$,
\[
\tau\bigl(x\alpha_h(y)\bigr)\longrightarrow 0
\qquad\text{as }h\to\infty\text{ in }H.
\]
Then
\[
M_\tau^H=\C.
\]
\end{proposition}

\begin{proof}
Let $x\in M_\tau^H$.
Replacing $x$ by $x-\tau(x)1$, we may assume that $\tau(x)=0$.
Then for every $h\in H$,
\[
\tau\bigl(x\alpha_h(x^*)\bigr)=\tau(xx^*).
\]
By the mixing hypothesis the left-hand side tends to $0$ as $h\to\infty$ in $H$.
Hence $\tau(xx^*)=0$, and faithfulness of $\tau$ gives $x=0$.
Thus $M_\tau^H=\C$.
\end{proof}

\begin{corollary}
\label{cor:mixing-relative}
Assume that $\T(A)^G\neq\varnothing$.
Let $H\leq G$ be infinite.
Assume that the pair $(G,H)$ has relative property~\textup{(T)} and that the $H$-mixing hypothesis of Proposition~\ref{prop:mixing} holds for every $\tau\in \ext \T(A)^G$.
Then $\T(A)^G$ is a Bauer simplex.
\end{corollary}

\begin{proof}
Combine Proposition~\ref{prop:mixing} with Corollary~\ref{cor:main-relative}.
\end{proof}

\begin{remark}
\label{rem:mixing-dense}
In Proposition~\ref{prop:mixing} it is enough to verify the convergence assumption on a $\normtwo{\cdot}{\tau}$-dense $*$-subalgebra of $M_\tau$.
Indeed, the estimate
\[
\abs{\tau\bigl(x\alpha_h(y)\bigr)-\tau\bigl(x'\alpha_h(y')\bigr)}
\leq \normtwo{x-x'}{\tau}\,\normtwo{y}{\tau}+\normtwo{x'}{\tau}\,\normtwo{y-y'}{\tau}
\]
is uniform in $h\in H$ because each $\alpha_h$ preserves the $2$-norm.
\end{remark}

\section{Applications}
\label{sec:applications}

\subsection{Quasi-local permutation actions}

\begin{example}[Generalized Bernoulli actions]
\label{ex:generalized-bernoulli}
Let $I$ be a countable $G$-set and let $D$ be a separable unital $C^*$-algebra.
Set
\[
A=\bigotimes_{i\in I} D,
\]
with the spatial infinite tensor product taken with respect to the unit of $D$, and let $\beta\colon G\curvearrowright A$ be the permutation action.
For a finite subset $F\subseteq I$, let $A_F\subseteq A$ be the  copy of $\bigotimes_{i\in F} D$.
Then
\[
A_{\mathrm{loc}}:=\bigcup_{F\subseteq I\text{ finite}} A_F
\]
is a norm-dense $*$-subalgebra satisfying the hypotheses of Corollary~\ref{cor:quasi-local}.
Consequently, if $G$ has property~\textup{(T)} and the action of $G$ on $I$ displaces finite subsets, then every nonempty invariant trace simplex $\T(A)^G$ is Bauer.
In particular, if $G$ is infinite, taking $I=G$ with the left translation action yields Bernoulli shifts of property~\textup{(T)} groups.
\end{example}

\begin{proposition}[Product traces for Bernoulli actions]
\label{prop:bernoulli-product-extreme}
Let $G\curvearrowright I$ be an action on a countable set, and assume that
there exists an infinite subgroup $H\leq G$ such that, for every finite
subsets $F,K\subseteq I$, the set
\[
\set{h\in H: hK\cap F\neq \varnothing}
\]
is finite.
Let $D$ be a separable unital $C^*$-algebra with $\T(D)\neq\varnothing$, set
\[
A=\bigotimes_{i\in I} D,
\]
and let $\beta\colon G\curvearrowright A$ be the permutation action.
For $\rho\in \T(D)$, let
\[
\tau_\rho:=\bigotimes_{i\in I}\rho.
\]
Then
\(
\tau_\rho\in \ext \T(A)^G.
\)
More precisely, for all $x,y\in M_{\tau_\rho}$ with
$\tau_\rho(x)=\tau_\rho(y)=0$,
\[
\tau_\rho\bigl(x\,\beta_h(y)\bigr)\longrightarrow 0
\qquad\text{as }h\to\infty\text{ in }H,
\]
and hence
\(
M_{\tau_\rho}^H=\C.
\)
\end{proposition}

\begin{proof}
First of all, the trace $\tau_\rho$ is clearly $G$-invariant. Now, let $F,K\subseteq I$ be finite, and let $x\in A_F$, $y\in A_K$ satisfy
\(
\tau_\rho(x)=\tau_\rho(y)=0.
\)
Set
\[
E_{F,K}:=\set{h\in H: hK\cap F\neq \varnothing}.
\]
By hypothesis, $E_{F,K}$ is finite.
If $h\notin E_{F,K}$, then $\beta_h(y)\in A_{hK}$ and $hK\cap F=\varnothing$,
so the product trace factorizes across the commuting tensor factors and gives
\[
\tau_\rho\bigl(x\,\beta_h(y)\bigr)
=
\tau_\rho(x)\tau_\rho(\beta_h(y))
=
\tau_\rho(x)\tau_\rho(y)
=
0.
\]
Thus the required convergence holds for centered local elements.
By Remark~\ref{rem:mixing-dense}, it follows for all centered
$x,y\in M_{\tau_\rho}$.

Proposition~\ref{prop:mixing} now gives
\(
M_{\tau_\rho}^H=\C,
\)
hence
\(
M_{\tau_\rho}^G\subseteq M_{\tau_\rho}^H=\C.
\)
Therefore,
\[
Z(M_{\tau_\rho})^G=\C,
\]
and Corollary~\ref{cor:trace-extreme-criterion} then shows that
\(
\tau_\rho\in \ext \T(A)^G.
\)
\end{proof}

\begin{corollary}
\label{cor:bernoulli-trace-embedding}
Let $G$ be an infinite countable discrete group, let $D$ be a separable unital
$C^*$-algebra with $\T(D)\neq\varnothing$, and let
\[
A=\bigotimes_G D
\]
with the Bernoulli shift action of $G$.
Then the map
\[
\T(D)\longrightarrow \ext \T(A)^G,
\qquad
\rho\longmapsto \tau_\rho:=\bigotimes_G \rho,
\]
is a continuous injection.
In particular, its image is a compact subset of $\ext \T(A)^G$.

If, in addition, $G$ has property~\textup{(T)}, then $\T(A)^G$ is Bauer.
\end{corollary}

\begin{proof}
Apply Proposition~\ref{prop:bernoulli-product-extreme} with $I=G$ and with
$H=G$. For finite subsets $F,K\subseteq G$, one has
\[
\set{h\in G: hK\cap F\neq \varnothing}\subseteq FK^{-1},
\]
so the hypothesis is satisfied.

To see injectivity, let $e\in G$ be the identity element and, for $a\in D$,
let $a^{(e)}\in A$ denote the element with $a$ in the $e$-coordinate and $1$
in all other coordinates.
Then
\[
\tau_\rho(a^{(e)})=\rho(a),
\]
so $\rho\mapsto \tau_\rho$ is injective.

For continuity, observe that on every local elementary tensor
\[
a_1\otimes\cdots\otimes a_n
\]
one has
\[
\tau_\rho(a_1\otimes\cdots\otimes a_n)
=
\prod_{k=1}^n \rho(a_k),
\]
which depends continuously on $\rho$.
Since local tensors are norm-dense in $A$, the map is weak-$^*$ continuous.
Because $\T(D)$ is compact and $\T(A)^G$ is Hausdorff, the map is a
homeomorphism onto its image.

If $G$ has property~\textup{(T)}, Bauer-ness follows from
Example~\ref{ex:generalized-bernoulli}.
\end{proof}
\begin{example}[A Bauer simplex with a large extreme boundary]
\label{ex:bernoulli-poulsen-boundary}
Let $G$ be an infinite countable discrete group with property~\textup{(T)}, and set
\[
A=\bigotimes_G C^*(F_2)
\]
with the Bernoulli shift action.
By Corollary~\ref{cor:bernoulli-trace-embedding}, the simplex $\T(A)^G$ is Bauer
and its extreme boundary contains a compact subset homeomorphic to
\[
\T(C^*(F_2)).
\]
By \cite{OSV}, the latter is the Poulsen simplex.
Thus $\ext \T(A)^G$ contains a compact copy of the Poulsen simplex.
\end{example}

\begin{corollary}
\label{cor:direct-sum-group}
Let $I$ be a countable $G$-set and let $\Delta$ be a countable discrete group.
Set
\[
\Gamma:=\bigoplus_I \Delta,
\]
and let $G$ act on $\Gamma$ by permuting coordinates.
Assume that $G$ has property~\textup{(T)} and that the action on $I$ displaces finite subsets.
Then the simplex of $G$-invariant traces on $C^*(\Gamma)$ is Bauer.
The same conclusion holds for $C_r^*(\Gamma)$.
\end{corollary}

\begin{proof}
We treat $C^*(\Gamma)$; the reduced case is identical.
For each finite $F\subseteq I$, let $\Gamma_F\leq \Gamma$ be the subgroup of elements supported in $F$, and let
\[
A_F:=\C[\Gamma_F]\subseteq \C[\Gamma]\subseteq C^*(\Gamma)
\]
be its group algebra.
Since every element of $\Gamma$ has finite support,
\[
A_{\mathrm{loc}}:=\bigcup_{F\subseteq I\text{ finite}} A_F=\C[\Gamma]
\]
is norm-dense in $C^*(\Gamma)$.
Moreover, $\alpha_g(A_F)=A_{gF}$ for every $g\in G$, and if $F\cap K=\varnothing$, then the subgroups $\Gamma_F$ and $\Gamma_K$ commute elementwise, so $[A_F,A_K]=0$.
Thus Corollary~\ref{cor:quasi-local} applies.
\end{proof}

\subsection{Invariant traces on group \texorpdfstring{$C^*$}{C-star}-algebras}

The next corollary recovers the invariant group-trace theorem from \cite[Thm 1.9]{LSV} in the language developed here.

\begin{corollary}[ Invariant traces on groups]
\label{cor:group-traces}
Let $\Gamma$ and $\Lambda$ be countable discrete groups, and let $\theta\colon \Lambda\to \Aut(\Gamma)$ be an action.
Assume that
\[
\Inn(\Gamma)\leq \theta(\Lambda)
\]
and that $\Lambda$ has property~\textup{(T)}.
Then the simplex of $\Lambda$-invariant traces on $C^*(\Gamma)$ is Bauer.
The same conclusion holds for $C_r^*(\Gamma)$.
\end{corollary}

\begin{proof}
We treat $C^*(\Gamma)$; the reduced case is identical.
Let $A=C^*(\Gamma)$ and let $\alpha$ be the induced action of $\Lambda$ on $A$.
Fix $\tau\in \ext \T(A)^\Lambda$ and write $M_\tau=\pi_\tau(A)''$.
Let $u_\gamma\in A$ denote the canonical unitary corresponding to $\gamma\in \Gamma$.

If $x\in M_\tau^\Lambda$, then $x$ is fixed by every inner automorphism of $\Gamma$, because $\Inn(\Gamma)\leq \theta(\Lambda)$.
Equivalently,
\[
\pi_\tau(u_\gamma)x\pi_\tau(u_\gamma)^*=x
\qquad (\gamma\in \Gamma).
\]
Hence $x$ commutes with every $\pi_\tau(u_\gamma)$.
Since these unitaries generate $M_\tau$, it follows that $x\in Z(M_\tau)$.
By Remark~\ref{rem:central-detection}, we obtain $M_\tau^\Lambda=\C$.
Corollary~\ref{cor:main-trace} now applies.
\end{proof}

\subsection{Reduced crossed products}

Our crossed-product applications rely on certain trace-extension theorems.
Ursu gives a general description of traces on crossed products \cite{Ursu}.
For our purposes, we only need the following completely formal consequence of
uniqueness of tracial extension.

Let $B_\star=A\rtimes_{\alpha,\star} G$ for $\star\in\{r,u\}$, and let
\[
E_\star\colon B_\star\to A
\]
denote the canonical conditional expectation.

\begin{proposition}[Uniqueness principle]
\label{prop:crossed-homeo}
Assume that every $\tau\in \T(A)^G$ has a unique tracial extension to $B_\star$.
Then restriction induces an affine homeomorphism
\[
\mathrm{res}\colon \T(B_\star)\to \T(A)^G,
\qquad
\mathrm{res}(\varphi)=\restr{\varphi}{A},
\]
whose inverse is
\[
\mathrm{ext}\colon \T(A)^G\to \T(B_\star),
\qquad
\mathrm{ext}(\tau)=\tau\circ E_\star.
\]
\end{proposition}

\begin{proof}
If $\varphi\in \T(B_\star)$, then $\restr{\varphi}{A}\in \T(A)^G$: it is a trace
on $A$, and for $a\in A$ and $g\in G$,
\[
\restr{\varphi}{A}(\alpha_g(a))
=
\varphi(u_gau_g^*)
=
\varphi(a).
\]

Conversely, let $\tau\in \T(A)^G$.
We claim that $\tau\circ E_\star$ is tracial on $B_\star$.
Since the algebraic crossed product is dense, it suffices to check this on
monomials.
For $a,b\in A$ and $g,h\in G$,
\[
(a u_g)(b u_h)=a\alpha_g(b)u_{gh}.
\]
If $gh\neq e$, then
\[
(\tau\circ E_\star)\bigl((a u_g)(b u_h)\bigr)
=
(\tau\circ E_\star)\bigl((b u_h)(a u_g)\bigr)
=
0.
\]
If $h=g^{-1}$, then
\begin{align*}
(\tau\circ E_\star)\bigl((a u_g)(b u_{g^{-1}})\bigr)
&=\tau\bigl(a\alpha_g(b)\bigr)\\
&=\tau\bigl(\alpha_{g^{-1}}(a)b\bigr)\\
&=\tau\bigl(b\alpha_{g^{-1}}(a)\bigr)\\
&=(\tau\circ E_\star)\bigl((b u_{g^{-1}})(a u_g)\bigr),
\end{align*}
using $G$-invariance and traciality of $\tau$.
Thus $\tau\circ E_\star\in \T(B_\star)$.

The maps $\mathrm{res}$ and $\mathrm{ext}$ are therefore well-defined,
continuous, and affine, and clearly
\[
\mathrm{res}\circ \mathrm{ext}=\id_{\T(A)^G}.
\]
By uniqueness, every $\varphi\in \T(B_\star)$ is the unique tracial extension of
$\restr{\varphi}{A}$, so
\[
\varphi=\restr{\varphi}{A}\circ E_\star=\mathrm{ext}(\mathrm{res}(\varphi)).
\]
Hence $\mathrm{res}$ and $\mathrm{ext}$ are inverse affine homeomorphisms.
\end{proof}

For reduced crossed products, the uniqueness hypothesis is supplied by
Bryder--Kennedy when the amenable radical is trivial.
Thus the next statement is precisely the correspondence they prove.

\begin{corollary}[Bryder--Kennedy]
\label{cor:reduced-homeo-radical}
Assume that $\Rad(G)=\{e\}$.
Then restriction induces an affine homeomorphism
\[
\T(A\rtimes_{\alpha,r} G)\cong \T(A)^G.
\]
In particular, $A\rtimes_{\alpha,r} G$ has Bauer trace simplex if and only if
$\T(A)^G$ does.
\end{corollary}

\begin{proof}
By \cite[Theorem~1.3 and Corollary~1.4]{BryderKennedy}, every tracial state
$\varphi$ on $A\rtimes_{\alpha,r} G$ satisfies
\[
\varphi=\restr{\varphi}{A}\circ E_r.
\]
Equivalently, every $G$-invariant trace on $A$ has a unique tracial extension to
$A\rtimes_{\alpha,r} G$.
Proposition~\ref{prop:crossed-homeo} now applies.
\end{proof}

We now combine
Corollary~\ref{cor:reduced-homeo-radical} with the Bauer results for invariant
trace simplices proved above.

\begin{corollary}[Permutational wreath products]
\label{cor:wreath}
Let $I$ be a countable $G$-set and let $\Delta$ be a countable discrete group.
Assume that $G$ has property~\textup{(T)}, that $\Rad(G)=\{e\}$, and that the action of $G$ on $I$ displaces finite subsets.
Then the reduced group $C^*$-algebra
\[
C_r^*\bigl((\bigoplus_I \Delta)\rtimes G\bigr)
\]
has Bauer trace simplex.
In particular, if $G$ is infinite and $I=G$ with the left translation action, then
\[
C_r^*(\Delta\wr G)
\]
has Bauer trace simplex.
\end{corollary}

\begin{proof}
Set $\Gamma=\bigoplus_I \Delta$.
By Corollary~\ref{cor:direct-sum-group}, the simplex $\T(C_r^*(\Gamma))^G$
of $G$-invariant traces is Bauer.
Since $\Rad(G)=\{e\}$, Corollary~\ref{cor:reduced-homeo-radical} identifies
\[
\T\bigl(C_r^*(\Gamma)\rtimes_r G\bigr)
\]
affinely with $\T(C_r^*(\Gamma))^G$.
Using the canonical isomorphism
\[
C_r^*(\Gamma\rtimes G)\cong C_r^*(\Gamma)\rtimes_r G,
\]
we conclude that $C_r^*((\bigoplus_I \Delta)\rtimes G)$ has Bauer trace simplex.
The final statement is the special case $I=G$, for which the left translation action displaces finite subsets as soon as $G$ is infinite.
\end{proof}

We conclude with a semidirect-product consequence.

\begin{corollary}
\label{cor:semidirect}
Let $\Gamma$ and $\Lambda$ be countable discrete groups, and let $\theta\colon \Lambda\to \Aut(\Gamma)$ be an action.
Assume that
\[
\Inn(\Gamma)\leq \theta(\Lambda),
\qquad
\Rad(\Lambda)=\{e\},
\]
and that $\Lambda$ has property~\textup{(T)}.
Then the reduced group $C^*$-algebra
\[
C_r^*(\Gamma\rtimes_\theta \Lambda)
\]
has Bauer trace simplex.
\end{corollary}

\begin{proof}
There is a canonical isomorphism
\[
C_r^*(\Gamma\rtimes_\theta \Lambda)\cong C_r^*(\Gamma)\rtimes_{\theta,r}\Lambda.
\]
By Corollary~\ref{cor:group-traces}, the invariant trace simplex $\T(C_r^*(\Gamma))^\Lambda$ is Bauer.
Since $\Rad(\Lambda)=\{e\}$, Corollary~\ref{cor:reduced-homeo-radical} identifies
\[
\T\bigl(C_r^*(\Gamma)\rtimes_{\theta,r}\Lambda\bigr)
\]
affinely with $\T(C_r^*(\Gamma))^\Lambda$.
Hence $C_r^*(\Gamma\rtimes_\theta \Lambda)$ has Bauer trace simplex.
\end{proof}

\bibliographystyle{plain}
\bibliography{refs}

\end{document}